\documentclass[reqno,12pt]{amsart}

\usepackage{amsmath}
\usepackage{amsfonts}
\usepackage{amssymb}
\usepackage{eufrak}
\usepackage{amscd}
\usepackage{amsthm}
\usepackage{amstext}
\usepackage[all]{xy}

\newcommand{\Ad}{\operatorname{Ad}}

\newcommand{\id}{\operatorname{id}}

   \theoremstyle{plain}
   \newtheorem{thm}{Theorem}
   \newtheorem{prop}[thm]{Proposition}
   \newtheorem{lem}[thm]{Lemma}
   \newtheorem{cor}[thm]{Corollary}
   \theoremstyle{definition}
   
   \newtheorem{defn}[thm]{Definition}

   \newtheorem{remark}[thm]{Remark}

\author{A. Korchagin}

\date{}

\address{Moscow State University,
Leninskie Gory, Moscow,
119991, Russia}

\email{mogilevmedved@yandex.ru}

\thanks{The author acknowledges partial support by the RFBR grant No. 14-01-00007.}

\title{MF-property for countable discrete groups}

\begin{document}
\maketitle

\begin{abstract}
In this article we study MF-property for countable discrete groups, i.e. groups which admit embedding into unitary group of $C^*$-algebra $\prod M_n/\oplus M_n$. We prove that Baumslag group $\langle a,b|a^{a^b}=a^2\rangle$ has MF-property and check some permanent facts about MF-groups.

\end{abstract}

\section{Introduction}

By definition MF-groups are countable groups which admit embedding into $U(\prod_{n=1}^\infty M_n/\oplus_{n=1}^\infty M_n)$ where $M_n$ - algebra of $n$-by-$n$ complex-valued matrices, $U(A)$ - group of unitary elements of $C^*$-algebra $A$. MF-groups were first considered in [7], where it was proved that for amenable groups MF-property equivalent to quasidiagonality of $C^*_r(G)$. In [7] it was also proved that LEF-groups have MF-property. The main motivation for considering and studying amenable MF-groups is famous conjecture that for amenable group $G$ algebra $C^*_r(G)$ is quasidiagonal. Recently this conjecture was proved in [17], i.e. all countably amenable groups are MF. So it is very natural to examine what non-amenable groups have this property. The first reason is the connection with vector bundles. From [14] we know how to construct a vector bundle on $BG$ from homomorphism $G\to U(\prod M_n/\oplus M_n)$ (such homomorphisms are usually called asymptotic homomorphisms or MF-representation). This vector bundle has some good properties which make it similar to vector bundle which is constructed from finitedimensional representation of group $G$. But on the other hand this construction often produce  whole $K^0(BG)$, which is in some sense finitedimensional way to Novikov conjecture.

Our second motivation - is connection between MF-property and hyperlinearity. It is easy to see that MF-property for $G$ is equivalent to possibility of embedding $G\hookrightarrow \prod U_n/\oplus U_n$, where $U_n$ is usual $n$-unitary group and $\oplus U_n=\{u_n\in U_n: \|u_n-1\|\to 0\}$. Recall that countable group $G$ is called hyperlinear if it admits embedding $G\hookrightarrow\prod U_n/\oplus_2 U_n$, where $\oplus_2 U_n=\{u_n\in U_n:\|u_n-1\|_2\to 0\}$ and $\|a\|_2=\sqrt{\tau(a^*a)}$ and $\tau$ - normalized trace on $M_n$
(it seems to be interesting to consider more general approximation in $GL_n$ instead of $U_n$. Apriori we get another classes of groups after replacing $U_n$ by $GL_n$ in definition of hyperlinear or MF-group. Concept of linear sofic groups is example of such generalizations, see [1] in this direction).
Due to similarity of definitions we can check some facts about MF-groups just rewriting proof of similar facts about hyperlinear groups into MF-language, but not always we can do that. For example, it is known that amalgamated product of two hyperliner group over amenable group us hyperlinear; while we don`t even know is it true that amalgamated product of two MF-groups over finite group is MF-group. While we have $\|a\|_2\leq \|a\|$ and so every MF-representation is hyperlinear representation, faithful MF-representation could be non-faithful hyperlinear representation and MF-property does not automatically imply hyperlinearity. Now nobody knows example of non hyperlinear group (existence of non MF-groups is also open question), but one of the main candidates for this role is Higman group $\langle a,b,c,d|a^b=a^2,b^c=b^2,c^d=c^2,d^a=d^2\rangle$. The famous property of this group is nonexistence of nontrivial finitedimensional representation and it is not clear, how we can construct nontrivial MF-representation or hyperlinear representation. There is homomorphisms $\langle a,b,c,d|a^b=a^2,b^c=b^2,c^d=c^2,d^a=d^2\rangle\to G=\langle x,y|x^{x^y}=x^2,[x,y^4]=1\rangle$ defined via $a\mapsto x, b\mapsto x^y,c\mapsto x^{y^2},d\mapsto x^{y^3}$. So if we have nontrivial represenation of group G - it is a great chance to construct a representation of Higman group, but group G is too complicated and for partial understanding its representations we consider less complicated Baumslag group $\langle x,y,|x^{x^y}=x^2\rangle$.

\section{Permanent facts}

We will consider only countable groups, all maps between groups are assumed to be unital. We also use notation $a^b=b^{-1}ab$.

\begin{defn}
Countable group G is called MF-group (or has MF-property) if there is injective homomorphism $G\hookrightarrow\prod U_n/\oplus U_n$, where $\oplus U_n=\langle\{u_n\}\in\prod U_n:\|u_n-1\|\to 0\rangle$.
\end{defn}

\begin{prop}
The following conditions are equivalent

1) G is MF-group.

2) There are maps $\alpha_n:G\to U_n$ such that for every $g,h\in G$ we have $\|\alpha_n(gh)-\alpha_n(g)\alpha_n(h)\|\to 0$ and for every $g\neq 1$ we have $\|\alpha_n(g)-1\|\nrightarrow 0$

3) For every finite set $F\subset G$ there is $\delta$ such that for every $\varepsilon>0$ there is $n$ and map $\alpha:G\to U_n$ such that for every $g,h\in F$ we have $\|\alpha(gh)-\alpha(g)\alpha(h)\|<\varepsilon$ and for every $g\in F$ such that $g\neq 1$ we have $\|\alpha(g)-1\|>\delta$

4) For some subsequense $\{n_k\}$ we have inclusion $G\hookrightarrow \prod U_{n_k}/\oplus U_{n_k}$.

5) $G\hookrightarrow U(\prod M_n/\oplus M_n)$.

6) $G\hookrightarrow U(A)$ for some MF-algebra $A$.

\end{prop}

\begin{proof}
Easy exercise.
\end{proof}

\begin{defn}
We will call homomorphism $\alpha:G\to \prod U_n/\oplus U_n=U(\prod M_n/\oplus M_n)$ by asymptotic homomorphisms (or MF-representation). We will call maps $\alpha_n$ (which appear from some lift $G\to \prod U_n$ of $\alpha$) by almost representations.
\end{defn}

There is very important homomorphism $U_n\to U_{n^2}$, $u\mapsto \Ad(u)$, where $\Ad(u)$ is unitary matrix of conjugate by $u$ in the space $\mathbb{C}^{n^2}=M_n(\mathbb{C})$ with inner product $\langle A,B\rangle=\sum_{i,j}\bar{A}_{i,j}B_{i,j}$.

\begin{prop}
For every $\delta>0$ there is $k_\delta\in\mathbb{N}$ such that for every $u\in U_n$ with $diam(\sigma(u))>\delta$ there is $k<k_\delta$ such that $\|\gamma^k(u)-1\|\geq\sqrt{2}$.
\end{prop}

\begin{proof}
Put $k_\delta=[\log_2(\frac{\pi}{\delta})]+1$. We may assume that $u=diag(e^{i\alpha_1},...,e^{i\alpha_n})$. We have $\gamma(u)=diag\{e^{i(\alpha_i-\alpha_j)}\}$ in the basis consisting of matrix units in $M_n(\mathbb{C})=\mathbb{C}^{n^2}$. Let $x,y$ be arguments of eigenvalues of $u$ such that $|e^{ix}-e^{iy}|=diam(\sigma(u))$. So, we have $e^{i(x-y)}$ and $e^{i(y-x)}$ among eigenvalues of $\gamma(u)$. By induction we have $e^{i2^k(x-y)}$ and $e^{i2^k(y-x)}$ among eigenvalues of $\gamma^k(u)$. Consider minimal $k$ such that $2^k(x-y)\in[\frac\pi 2,\pi)$. In this case we have $|e^{i2^k(x-y)}-1|\geq\sqrt{2}$ and so $\|\gamma^{k+1}(u)-1\|\geq\sqrt{2}$. It is easy to see that $k<k_{diam(\sigma(u))}$.
\end{proof}

Later for operator we will use standard notation $x=_\varepsilon y$ when $\|x-y\|<\varepsilon$.

\begin{prop}
Let $\alpha^\prime:G\to\prod U_n/\oplus U_n$ - some MF-representation such that $\alpha^\prime(g)\neq 1$ for some $g\in G$. Then there is MF-representation $\beta$ such that $\|\beta(g)-1\|\geq\sqrt{2}$.
\end{prop}

\begin{proof}
For convenience we will consider MF-representation $\alpha^\prime$ as set $\{\alpha_n^\prime\}$ of almost representation. Putting $\alpha_n=\alpha_n^\prime\oplus 1$ we have for every $g\in G$ that $1\in\sigma(\alpha_n(g))$. Since $\alpha(g)\neq 1$, there is $\delta>0$ and $n_0$ such that for every $n>n_0$ we have $\|\alpha_n(g)-1\|>\delta$ (if it is necessary we consider some subsequence). Since $1\in\sigma(\alpha_n(g))$ we have $diam(\sigma(\alpha_n(g)))>\delta$. By Proposition 4 we can construct $k_\delta$ and numbers $k(n)$ with $k(n)<k_\delta$ such that for every $n>n_0$ we have $\|\beta_n(g)-1\|\geq\sqrt{2}$ where $\beta_n=\gamma^{k(n)}\circ\alpha_n$.
Moreover $\{\beta_n\}$ is asymptotic homomorphism. Indeed, for every $\varepsilon>0$ and every finite set $K\subset G$ we can ensure that $\alpha_n(q)\alpha_n(h)\alpha_n^{-1}(qh)=_\varepsilon 1$ for every $q,h\in K$ and $n$ big enough. Since $\gamma^k$ -homomorphism then
$\beta_n(q)\beta_n(h)\beta_n^{-1}(qh)=\gamma^{k(n)}(\alpha_n(q)\alpha_n(h)\alpha_n^{-1}(qh))$. From the formula $\gamma(diag(\{e^{i\alpha_i}\}))=diag(\{e^{i(\alpha_i-\alpha_j)}\})$ it is easy follows that $\|\gamma(u)-1\|\leq 2\|u-1\|$, so $\beta_n(q)\beta_n(h)\beta_n^{-1}(qh)=_{\varepsilon 2^{k_\delta}}1$. It means that $\beta$ - asymptotic representation.
\end{proof}

\begin{prop}
Residually MF-group is MF-group.
\end{prop}

\begin{proof}
Let $G$ - residually MF-group. It means that for every $g\neq 1$ there is MF-homomorphisms $\alpha^g$ (here $g$ is index) such that $\alpha^g(g)\neq 1$. By Proposition 5 we can find some another MF-homomorphism $\beta^g$ such that $\|\beta^g(g)-1\|\geq\sqrt{2}$. Consider $\varepsilon>0$ and finite set $K\subset G$ and let $\beta^K=\oplus_{g\in K}\beta^g$. As finite direct sum of MF-representation $\beta^K$ is also MF-representation, so for some big enough $N=N(K,\varepsilon)$ we have $\beta^K_N(gh)=_\varepsilon\beta^K_N(g)\beta^K_N(h)$ for every $g,h\in K$ and $\|\beta^K_N(g)-1\|\geq 1$ for every $g\in K$ such that $g\neq 1$.

Consider our group $G=\cup K_n$ as union os increasing sequence of finite sets. Then $\omega_n=\beta^{K_n}_{N(K_n,\frac1n)}$ is faithful asymptotic representation. It is easy to see that $\|\omega(g)-1\|\geq 1$ for every $g\neq 1$, and for every $\varepsilon$ and every finite $K\subset G$ we can find $n$ such that $\varepsilon>\frac1n$ and $K\subset K_n$. So $\|\omega_n(g)\omega_n(h)-\omega_n(gh)\|<\varepsilon$ for every $g,h\in K$.
\end{proof}

Proofs of Propositions 4-6 are almost the same as in the case of hyperlinear groups (see [8]).

Using $C^*$-theory we can present shorter proof of Proposition 6. Let $G$ -residually MF-groups, so there is MF-algebras $A_n$ such that $G\hookrightarrow\prod A_n$ (where image lies in unitary group). Consider $B_n=A_1\oplus...\oplus A_n$, homomorphism $\alpha:\prod A_n\hookrightarrow \prod B_n$ defined via $\alpha(\{a_n\})=\{a_1\oplus...\oplus a_n\}$. Consider composition map $\beta:G\hookrightarrow\prod A_n\hookrightarrow B_n\to\prod B_n/\oplus B_n$. It is easy to see that $\beta$ - injective and $C^*(\beta(G))$ is separable subalgebra of $\prod B_n/\oplus B_n$. So $C^*(\beta(G))$ is MF-algebra by [5].

\begin{prop}
Let $G$ - group. Then $G$ is MF-group iff for every $\varepsilon>0$ and every finite $K\subset G$ there is map $\alpha:G\to U_n$ to some finitedimensional unitary group such that for every $g,h\in K$ inequality $\|\alpha(gh)-\alpha(g)\alpha(h)\|<\varepsilon$ holds and for every nontrivial $g\in G$ we have $\|\alpha(g)-1\|\geq\sqrt{2}$.
\end{prop}

\begin{proof}
$2\Rightarrow 1$ is obvious. For $1\Rightarrow 2$ we can consider $\alpha=\beta^K_{N(K,\varepsilon)}$ from the proof of Proposition 6.
\end{proof}

Proposition 7 means that for MF-groups there is injective MF-homomorphism $\alpha$ such that $\alpha(G)$ is discrete in the induced topology of $U(\prod M_n/\oplus M_n)$.

\begin{prop}
Let $G=\langle a_1,...,a_n|r_1,...,r_m=1\rangle$ be finite presented group. Then $G$ is MF-group iff for every $\varepsilon>0$ and every finite set $K\subset G$ there is $N\in\mathbb{N}$ and matrices $A_1,...,A_n\in U_N$ such that

1) For every nontrivial $k\in F$ there is some corresponding word $\omega_k\in\mathbb{F}_n=\langle a_1,...,a_n\rangle$ in free group (i.e. for natural quotient homomorphism $\pi:\mathbb{F}_n\to G$ we have $\pi(\omega_k)=k$. For convenience we also assume $\omega_{a_j}=a_j$) such that $\|\omega_k(A_1,...,A_n)-1\|\geq 1$

2) $r_j(A_1,...,A_n)=_\varepsilon 1$
\end{prop}

\begin{proof}
Let us prove $\Leftarrow$. Let $\varepsilon>0$ and $F\subset G$. Consider matrices $A_1,...,A_n\in U_N$ corresponding to $\varepsilon$ and finite set $K=F\cdot F=\{gh:g,h\in F\}$. Define our almost representation $\alpha:G\to U_N$ in such way: $\alpha(k)=\omega_k(A_1,...,A_N)$ (for nontrivial $k\in K$) and arbitrary on $G\backslash K$. Let $C_{g,h}$ - minimal number of operations of type $r_j\leftrightarrow 1$, $a_j^{-1}a_j\leftrightarrow 1$, $a_ja_j^{-1}\leftrightarrow 1$ which is necessary to transform $\omega_{gh}$ into $\omega_g\omega_h$ (since $\pi(\omega_{gh})=\pi(\omega_g\omega_h)$ this transformation is possible). Put $C=\max_{g,h\in F} C_{g,h}$. Then it is easy to see that $\|\alpha(gh)-\alpha(g)\alpha(h)\|<C\varepsilon$ for all $g,h\in F$. These almost representations $\{\alpha\}$ generate faithful asymptotic representation because $\|\omega_k(A_1,...,A_n)-1\|\geq 1$ for nontrivial $k$ by assumption.

Let us prove $\Rightarrow$. Let $G$ - MF-group and $\alpha_N:G\to U_N$ faithful asymptotic representation. Using Proposition 7 we may assume that $\|\alpha_N(g)-1\|\geq\sqrt{2}$ for all $N$ and $g\neq 1$. Put $\omega_k\in\mathbb{F}_n$ arbitrary with property $\pi(\omega_k)=k$. As $\|\alpha_N(gh)-\alpha_N(g)\alpha_N(h)\|\to 0$ we have $\|\alpha_N(\omega_g(a_1,...,a_n))-\omega_g(\alpha_N(a_1),...,\alpha_N(a_n))\|\to 0$ as $N\to\infty$. Since $\newline{\limsup_N\|\alpha_N(\omega_g(a_1,...,a_n))-1\|\geq\sqrt{2}}$ we have $\|\omega_g(\alpha_N(a_1),...,\alpha_N(a_n))-1\|\geq 1$ for every nontrivial $g\in F$ and $N$ big enough. Analogously using $\|\alpha_N(gh)-\alpha_N(g)\alpha_N(h)\|\to 0$ we deduce $\|r_j(\alpha_N(a_1),...,\alpha_N(a_n))-1\|<\varepsilon$ for every $j$ and $N$ big enough. It means that there is some number $N_0$ such that matrices $A_i=\alpha_{N_0}(a_i)$ satisfy properties 1) and 2).
\end{proof}

As for amenable groups MF-property equivalent to quasidiagonality of $C^*(G)$ (see [7]) we have the following important version of main theorem from [17]:

\begin{thm}
Amenable groups are MF-groups.
\end{thm}

Using Proposition 6 we immediately deduce the following fact:

\begin{cor}
Residually MF-groups are MF-groups.
\end{cor}

This corollary covers very wide class of countable groups, for example Baumslag-Solitar groups $B(n,m)=\langle a,b| b^{-1}a^nb=a^m\rangle$ (from [12] we know that $B(n,m)^{\prime\prime}=\mathbb{F}$. This imply that group $B(n,m)$ is residually solvable and so residually amenable because solvable groups are amenable (see [4, Example 2.6.5])). Another way for checking MF-property for Baumslag-Solitar groups $\langle a,b| b^{-1}a^nb=a^m\rangle$ is direct following the proof of the main theorem of [16] with using [13] for appearing amalgamated products. This works because all approximations in [16] are norm approximations. 

The easiest example of non-residually solvable group is Baumslag group $\langle a,b| a^{a^b}=a^2\rangle$. This group is also MF-group and last paragraph is devoted to proof of this fact. We do not know is this group is residually amenable.

\begin{prop}
Let $G_1,G_2,...$ - MF-groups. Then $\bigoplus G_j$ also MF-group.
\end{prop}

\begin{proof}
If $\alpha_j$ are $(\varepsilon, F_j)$-almost representations of group $G_j$ then $\alpha_1\oplus...\oplus \alpha_m$ is $(\varepsilon, F_1\oplus...\oplus F_m)$-almost representation of group $G_1\oplus...\oplus G_m$. It is easy to see that every finite subset of $\bigoplus G_j$ is consisted in some finite direct subproduct.
\end{proof}

\begin{prop}
Let $G$ be MF-group, $F$ - finite group. Then $G\rtimes F$ is MF-group.
\end{prop}

Remark that we can prove only this weak permanent fact about cross product. We do not know answer also in the case of $K=\mathbb{Z}$, while in the case of hyperlinear groups it is true that if $G$ is hyperlinear and $F$ is amenable then $G\rtimes F$ is also hyperlinear.

\begin{proof}
Let $\alpha_n:G\to U_n$ be faithful asymptotic representation and $\gamma_k(g)=k^{-1}gk$ for $k\in F,g\in G$. Put $\beta_n(g)=\oplus_k \alpha_n(k^{-1}gk)\in U_{n|K|}$ for $g\in G$ and let $\beta_n(k)\in U_{n|F|}$ be shift such that $\beta_n(k^{-1})(\oplus_h y_h)\beta_n(k)=\oplus_k y_{kh}$ for every $y_h\in U_n$, i.e. shift $\beta_n(k)$ move "$h$-block" to "$kh$-block". Now define $\beta_n:G\rtimes F\to U_{n|F|}$ on whole $G\rtimes F$ by the formula $\beta_n(gk)=\beta_n(g)\beta_n(k)$ for $k\in F,g\in G$.

Sine $\alpha_n$ is asymptotic representation and $\beta_n(k_1k_2)=\beta_n(k_1)\beta_n(k_2)$ for every $k_1,k_2\in F$ we have $\beta_n(g_1k_1g_2k_2)=\beta_n(g_1\gamma_{k_1^{-1}}(g_2)k_1k_2)=\beta_n(g_1\gamma_{k_1^{-1}}(g_2))\beta_n(k_1)\beta_n(k_2)=_{o(1)}$
$\newline=_{o(1)}\beta_n(g_1) \beta_n(\gamma_{k_1^{-1}}(g_2))\beta_n(k_1)\beta_n(k_2)=\beta_n(g_1)\beta_n(k_1)\beta_n(g_2)\beta_n(k_2)=\beta_n(g_1k_1)\beta_n(g_2k_2)$. So $\beta_n$ is asymptotic representation of group $G\rtimes F$.

Let us deduce faithfulness of $\beta_n$. Consider arbitrary nontrivial $\omega\in G\rtimes F$. It has the following form $\omega=gk$ for some $g\in G,k\in F$. If $k=1$ then $\|\beta_n(\omega)-1\|\geq\|\alpha_n(g)-1\|\geq\sqrt{2}$. Consider case $k\neq 1$. As $\beta_n(k)$ is shift we get that unitary matrix $\beta_n(\omega)$ has only zeros on diagonal . It follows that $\|\beta_n(\omega)-1\|\geq 1$, i.e. $\beta_n$ is faithful asymptotic representation and so $G\rtimes F$ is MF-group.
\end{proof}

\begin{prop}
Let $G_j$ be MF-groups. Then $G=\underset{\longrightarrow}{\lim}G_j$ is also MF-group.
\end{prop}

We have only $C^*$-algebraic proof.

\begin{proof}
We have embeddings $\beta_j:G_j\hookrightarrow U(A_j)$ for some MF-algebras $A_j$. Moreover by Proposition 7 we may assume $\|\beta_j(g)-1\|\geq 1$ for every nontrivial $g\in G_j$. Let homomorphisms $\alpha_j^k:G_j\to G_{j+k}$ determine our direct limit. Define $\gamma_j:G_j\to\prod A_i/\oplus A_i$ by formula
$$\gamma_j(g)=(*,...,*,\beta_j(g),\beta_{j+1}(\alpha_j^1(g)),\beta_{j+2}(\alpha_j^2(g)),...)$$
where values of $*$ are not important. Since $\gamma_{j+k}\circ\alpha_j^k=\gamma_j$ we have that homomorphisms $\gamma_j$ define homomorphism $\gamma:\underset{\longrightarrow}{\lim}G_j\to \prod A_i/\oplus A_i$. Consider arbitrary nontrivial $g\in G$. For some $j$ we have $g\in Im G_j$ where $Im G_j$ - image of $G_j$ under natural map $G_j\to G$. It is easy to see that $\|\gamma(g)-1\|=\limsup_k\|\beta_{j+k}(\alpha_j^{k}(g))-1\|\geq 1$ since $g$ is nontrivial. So $\gamma$ is injective. Let $A_G$ be $C^*$-algebra generating by $\gamma(G)$. It is separable and $A_G\hookrightarrow \prod A_i/\oplus A_i$ so it is MF-algebra (see [4]). So $G$ is MF-group.
\end{proof}

\begin{prop}
Let $G,H$ be MF-groups. Then $G\star H$ is also MF-group.
\end{prop}

\begin{proof}
Consider injective homomorphisms $\alpha_G:g\hookrightarrow U(A_G)$ and $\alpha_H:H\hookrightarrow U(A_H)$ for some MF-algebras $A_G$ and $A_H$. We have $G\star H\overset{\alpha_G\star \alpha_H}\hookrightarrow A_G\star_0 A_H\overset\gamma\hookrightarrow A_G\star A_H$ where $A\star B$ - unital free product of $C^*$-algebras $A$ and $B$, $A\star_0 B$ - unital free algebraic product (i.e. product without completion) of $C^*$-algebras $A$ and $B$, homomorphism $\alpha_G\star\alpha_H$ is defined via formulas $\alpha_G\star\alpha_H(g)=\alpha_G(g)$ for $g\in G$ and $\alpha_G\star\alpha_H(h)=\alpha_H(h)$ for $h\in H$. Injectivity of $\alpha_G\star\alpha_H$ is obvious, injectivity of $\gamma$ follows from [2]. We know from [13] that unital free product of MF-algebras is MF-algebra, so $A_G\star A_H$ is MF-algebra and $G\star H$ is MF-group.
\end{proof}

It is known that if $G,H$ is hyperlinear groups and $K$ is amenable then $G\underset{K}{\star}H$ is hyperlinear group. But we do not know is it true that $G\underset{K}{\star}H$ is MF-group when $G,H$ is MF-groups and $K$ is finite group.

\begin{prop}
Let $\varphi_t:G\to U(B(\mathbb{H}))$ - pointwise continuous family of homomorphisms (i.e. for every $g\in G$ function $\varphi_t(g)$ is continuous) where $t\in [0,\infty)$ and $B(\mathbb{H})$ - algebra of bounded operators on some Hilbert space. Let $C\subset B(\mathbb{H})$ be some quasidiagonal algebra and

1) $\varphi_0$ is injective.

2) For every $g\in G$ there is some $c_g\in C$ such that $\varphi_t(g)\to c_g$ as $t\to\infty$.

Then $G$ is MF-group.
\end{prop}

\begin{proof}
Let $Q$ be countable dense subset of $[0,\infty)$ and $A$ be separable $C^*$-algebra generated by $\{\varphi_q(g)\}_{q\in Q,g\in G}$. Since $\varphi_t$ is pointwise continuous then $\varphi_t(g)\in A$ fot every $t\in [0,\infty)$. Consider $C^*$-algebra $\Omega=\{f\in C_b([0,\infty),A):f(\infty)\in C\}$ of continuous bounded $A$-valued function which tend to some element of $C$ at infinity. Remark that $C\subset A$. Define homomorphism $\varphi:G\to U(\Omega)$ via formula $\varphi(g)(t)=\varphi_t(g)$. This homomorphism is injective since $\varphi_0$ is injective. It is easy to see that $\Omega$ is homotopic to $C$, the construction of homotopy equivalence is following: $\alpha:C\to \Omega$ via $\alpha(c)(t)=c$ and $\beta:\Omega\to C$ via $\beta(f)=f(\infty)$. Obviously $\alpha\circ\beta\simeq\id_\Omega$ and $\beta\circ\alpha\simeq\id_C$. Since quasidiagonality is homotopy invariant (see [6], Theorem 7.3.6), algebra $\Omega$ is quasidiagonal, so algebra which is generated by $\varphi (G)$ is also quasidiagonal as subalgebra of quasidiagonal $\Omega$. Since separable quasidiagonal algebras are MF-algebras, we deduce that $G$ is MF-group.
\end{proof}

\section{Baumslag group}

In this section we prove that Baumslag group is MF-group. Its hyperlinearity follows from [10] (see also [4],[15]), where it is also proved that soficity is closed under extension by amenable groups (it is well known that sofic groups are hyperlinear).

We will use following notation:

$x^y=y^{-1}xy=\Ad_y x$

$B=\langle a,b|a^{a^b}=a^2\rangle$ - Baumslag group

$H=\langle a,b|a^b=a^2\rangle$

$H_j=\langle a_{-j},...,a_j|a_i^{a_{i+1}}=a_i^2,i=-j,...,j-1\rangle=H\underset{\mathbb{Z}}{\star}...\underset{\mathbb{Z}}{\star}H$ where multiplication factors of amalgamated product are numbered from $-j$ to $j-1$, $i$-th factor generated by $a_i$ and $a_{i+1}$ and the generator of common subgroup $\mathbb{Z}$ of $i$-th and $(i+1)$-th factors is $a_{i+1}$.

$H_\infty=\langle...a_{-j},...,a_j,....|a_i^{a_{i+a}}=a_i^2\rangle=\underset{\longrightarrow}{\lim}H_j$

$D_n=diag\{1,e^{\frac{2\pi}{n}},...,e^{\frac{2\pi (n-1)}{n}}\}$

$T_n$ - standard shift matrix in $M_n(\mathbb{C})$ i.e. $T_n=e_{1,2}+e_{2,3}+...+e_{(n-1),n}+e_{n,1}$ where $e_{i,j}$ - standard matrix units.

$U_n=U(M_n(\mathbb{C}))$

$max(K)$ for finite subset $K\subset H_\infty$ is minimal $j$ such that $K\subset H_j$. It easy to see that $max(K)$ also equals to maximal absolute values of such $j$ for which letter $a_j$ nonreduceable appears in words in $K$.

$x\sim y$ if $x$ and $y$ are unitary conjugate.

$x=_\varepsilon y$ if $\|x-y\|<\varepsilon$.

We will write $x\sim_\varepsilon y$ for unitary $x,y\in U_n$ if there is unitary $y^\prime\in U_n$ such that $\|y-y^\prime\|<\varepsilon$ and $x\sim y^\prime$. It means that after small perturbation $y$ become unitary conjugate to $x$.

\begin{prop}
There is isomorphism $B\simeq H_\infty\rtimes \mathbb{Z}$ where action of generator of $\mathbb{Z}$ on $H_\infty$ is defined via formula $a_i\mapsto a_{i+1}$.
\end{prop}

\begin{proof}
The proof is easy exercise.
\end{proof}

\begin{lem}
Consider automorphism $\varphi$ of group $G$, homomorphism $\alpha: G\rtimes_\varphi\mathbb{Z}\to F$ where $F$ - some group (possibly non-countable). Let for every nonzero $k\in\mathbb{Z}$ automorphism $\varphi^k$ be non-inner. We have that if $\alpha|_G$ is faithful then $\alpha$ is also faithful.
\end{lem}

\begin{proof}
Assume that $\alpha$ is not faithful. So we can find $g\in G$ and $z\in\mathbb{Z}$ for which $\alpha(gz)=1$ If $z=0$ than this contradict with faithfulness of $\alpha|_G$ (by $0$ we denote neutral element of $\mathbb{Z}$). So $z\neq 0$ and $\alpha(z)=\alpha(g^{-1})$. It means that for every $h\in G$ we have $\alpha(\varphi^z(h))=\alpha(\varphi\circ...\circ\varphi(h))=\alpha(z^{-1}hz)=\alpha(ghg^{-1})$. Since $\alpha|_G$ is faithful then $\varphi^z(h)=ghg^{-1}$, i.e. automorphism $\varphi^z$ is inner which is contradiction.
\end{proof}

\begin{prop}
All powers of automorphism $\varphi:H_\infty\to H_\infty$, $\varphi(a_j)=a_{j+1}$ are non-inner.
\end{prop}

\begin{proof}
Assume that for some nonzero $k\in\mathbb{Z}$ and $\omega\in H_\infty$ we have $\varphi^k(h)=\omega^{-1}h\omega$ for every $h\in H_\infty$. Put $N=max(\{\omega\})$. So $a_{N+k}=\varphi^k(a_N)=\omega^{-1}a_N\omega\in H_N$ since $\omega\in H_N$. But $a_{N+k}\notin H_N$. Contradiction.
\end{proof}

We need the following theorem from [3]:

\begin{thm}
Let $A,B$ be residually solvable groups, $D$ be solvable. Consider common subgroup $C\subset A,B$. If there is homomorphism $\beta:A\underset{C}{\star} B\to D$ such that $\beta|_C$ is faithful. Then $A\underset{C}{\star} B$ is residually solvable.
\end{thm}

As easy corollary we can deduce the following proposition:

\begin{prop}
For every $j$ group $H_j$ is residually solvable (so it is MF-group).
\end{prop}

\begin{proof}
Let us use induction to prove this fact.

Base case: $H$ is solvable and so residually solvable.

Inductive step: Group $H_j$ has $2j+1$ amalgamated  multiplication factors. Put $L_N=\langle a_{-j},...,a_N|a_i^{a_{i+1}}=a_i^2\rangle$. Trivially we have $L_{-j+1}=H$ and $L_j=H_j$. To prove proposition it is enough to show that if $L_N$ is residually solvable then
$L_{N+1}=L_N\underset{\mathbb{Z}}{\star} H$ is also residually solvable. By Theorem 19 it is enough to construct homomorphism $\beta: L_{N+1}\to H$ which is injective on common subgroup $\mathbb{Z}$ of $L_N$ and $H$. Define $\beta$ on generators in such way: $\beta(a_i)=1$ for $i<N$, $\beta(a_N)=a$, $\beta(a_{N+1})=b$. It is easy that this map extends to homomorphism, which satisfy necessary conditions.
\end{proof}

Remark, that neither $B$ nor $H_\infty$ can be residually solvable because residual solvability is closed under taking extensions by solvable groups (see [11]).

Let $f(N)=2^{p_N}-1$ where $p_N$ is $N$-th prime number.

\begin{prop}
There exists matrix $T$ such that $T^{-1}D_{f(N)}T=D^2_{f(N)}$ and \newline $T\sim 1\oplus D_{p_N}\oplus...\oplus D_{p_N}$. In other words spectrum of $T$ is set of all $p_N$-th roots of unity with the same multiplicity and additional $1$ with multiplicity $1$.
\end{prop}

\begin{proof}
We have $D_{f(N)}=diag(e^0,e^{1c},e^{2c},...,e^{(f(N)-1)c})$, $D_{f(N)}^2=diag(e^0,e^{2c},e^{4c},...,e^{2(f(N)-1)c})$, where $c=\frac{2\pi i}{f(N)}$. Since $f(N)$ is odd number then there exists bijection $\sigma$ between ${(0,1,2,...,f(N)-1)}$ and $(0,2,4,..., 2(f(N)-1))$ modulo $f(N)$. Consider matrix $T$ of permutation of basis vectors $e_j$ corresponding to permutation $\sigma$. It is easy to see that every disjoint $n$-cycle of $\sigma$ corresponds to set $\sqrt[n]{1}$ in spectrum of $T$ (because $T|_L$ is usual shift where $L=span\{T^ke_j\}_k$ for some $j\in (0,1,2,...,f(N)-1)$ which belongs to our disjoint cycle). To show that spectrum of $T$ has desired properties let us examine structure of disjoint cycles of $\sigma$. We have $0\mapsto 0$ - this trivial orbit corresponds to 1 with multiplicity 1 in spectrum. Let us prove that every nonzero $x$ has orbit of length $p_N$. As permutation $\sigma$ is defined via formula $x\mapsto 2x$ $(mod$ $ f(N))$ and since $(2^{p_N}-1)x= f(N)x=0$ $(mod$ $ f(N))$ we have that length of orbit divides $p_N$. But sine $p_N$ is prime number and orbit is nontrivial we have that length of every nontrivial orbit is $p_N$.
\end{proof}

\begin{prop}
For every $\varepsilon>0$, every finite set $K\subset H_\infty$ and every $j>max(K)$ there exists natural number $n$ and map $\varphi:H_\infty\to U_n$ such that:

1) $\|\varphi(kh)-\varphi(k)\varphi(h)\|<\varepsilon$ for every $k,h\in K$.

2) $\|\varphi(k)-1\|\geq 1$ for every nontrivial $k\in K$.

3) $\|\Ad_{\varphi(a_{i+1})}\varphi(a_i)-\varphi(a_i)^2\|<\varepsilon$ where $|i|\leq j$.

4) $\varphi(a_i)\sim_\varepsilon\varphi (a_l)$ for every $i,l$ with ${i},{l}\leq j+1$.
\end{prop}

\begin{proof}
As $K\subset H_j$ and there is no occurrence of elements of $H_\infty\backslash H_{j+1}$ in conditions 1)-4) then we can define $\varphi$ on $H_\infty\backslash H_{j+1}$ arbitrary. Since $H_{j+1}$ is MF-group then we can construct $\psi:H_{j+1}\to U_m$ which satisfy conditions 1)-3).

Idea is following: we construct asymptotic representation $\pi_N$ of $H_j$ such that spectrum of $\pi_N(a_i)$ would be uniform. Then $\varphi=\psi\oplus\pi_N\oplus...\oplus \pi_N$ has desired properties, because $\psi$-summand secures faithfulness conditions 2) and a lot of $\pi_N$-summands make spectrum of $\varphi(a_j)$ to be almost uniform, so guarantees condition 4).

Put $\pi_N(a_{-j-1})=D_{f(N)}$. On other generators generators of $H_{j+1}$ define $\pi_N$ by induction. Let we have already construct $\pi_N(a_i)$ such that $\pi_N(a_i)\sim D_{f(N)}$. By Proposition 21 we can find matrix $V$ such that $V^{-1}\pi_N(a_i)V=\pi_N(a_i)^2$ and spectrum of $V$ consists of $1$ with multiplicity 1 and $p_N$-th root of unity with multiplicity $\frac{2^{p_N}-2}{p_N}$ (the reason of using $f(N)$ instead of usual $N$ is that spectrum of $V$ is very simple. We do not know for odd $N$ good characterization of spectrum of matrix $R$ such that $R^{-1}D_NR=D_N^2$). It is easy to see that $D_{f(N)}\sim_{\frac{2}{p_N}} V$ (because we can  eigenvalue $e{^\frac{2\pi i k}{p_N}}$ with multiplicity $\frac{2^{p_N}-2}{p_N}$ uniformly spread on interval $(e^{\frac{2\pi i k}{p_N}},e^{\frac{2\pi i (k+1)}{p_N}})\subset S^1$. Then we should shift all eigenvalues to clear "space" for eigenvalue $1$ with multiplicity $1$). So we can matrix $S$ with properties $S=_{\frac{2}{p_N}} V$ and $S\sim D_{f(N)}\sim\pi_N(a_i)\sim...\sim \pi_N(a_{-j})$. Put $\pi_N(a_{i+1})=S$. It is clear that $\Ad_{\pi_N(a_{i+1})}\pi_N(a_i)=_{O(\frac{1}{p_N})}\pi_N(a_i)^2$ so $\pi_N$ is asymptotic representation and we can find $N$ large enough that almost representation $\pi_N$ satisfy conditions 1)-3).

Put $\varphi=\psi\oplus\pi_N\oplus...\oplus\pi_N$ where in direct sum there are $m$ direct $\pi_N$-summands, $N$ as in the previous paragraph, $m$ is dimension of $\psi$. Obviously $\phi$ satisfy conditions 1)-3). It is easy to see that $e^{i\lambda}\oplus D_n\sim_{\frac{1}{n}} D_{n+1}$ for every $\lambda$. Consider arbitrary generators $a_i$ and $a_l$ of $H_{j+1}$. Every normal matrix can be diagonalized so $\psi(a_{i})\sim e^{i\lambda_1}\oplus ...\oplus e^{i\lambda_m}$ è $\psi(a_l)\sim e^{i\mu_1}\oplus ...\oplus e^{i\mu_m}$ and $\phi(a_{i})\sim e^{i\lambda_1}\oplus ...\oplus e^{i\lambda_m}\oplus \pi_N(a_{i})\oplus... \oplus\pi_N(a_{i})\sim (e^{i\lambda_1}\oplus \pi_N(a_{i}))\oplus... \oplus(e^{i\lambda_m}\oplus \pi_N(a_{i}))\sim (e^{i\lambda_1}\oplus D_{f(N)})\oplus... \oplus (e^{i\lambda_m}\oplus D_{f(N)})\sim_{\frac{1}{f(N)}} D_{f(N)+1}\oplus... \oplus D_{f(N)+1}$. Similarly we can deduce $\phi(a_l)\sim_{\frac{1}{f(N)}} D_{f(N)+1}\oplus... \oplus D_{f(N)+1}$ and so $\phi(a_{i})\sim_\epsilon\phi(a_l)$.
\end{proof}

\begin{prop}
Let $u\in U_n$ and $\varepsilon>0$. Then there exist matrices $u_{-k},u_{-k+1},...,u_k\in U_n$ where $k=[\frac 1\varepsilon]$ with properties:

1) $u_0=1$.

2) $u_i=_{4\varepsilon} u_{i+1}$ for every $i$.

3) $u=u_{-k}u_n^{-1}$
\end{prop}

\begin{proof}
Put $u_i=1$ for $i\leq 0$ Minimal length of path between $1$ and $u^{-1}$ in the unitary group is not greater than 4 (because all unitary matrices can be diagonalized and so geodesic distance between $1$ and $u^{-1}$ is not greater than geodesic diameter of unit circle which is equal to $\pi$). So we can find matrices $u_0,u_1,...,u_k$ with $u_0=1$, $u_k=u^{-1}$ and $u_i=_{4\varepsilon} u_{i+1}$.
\end{proof}

\begin{thm}
Group $B=\langle a,b|a^{a^b}=a^2\rangle$ is MF-group.
\end{thm}

\begin{proof}
By Lemma 17 it is enough to construct asymptotic representation of $B$ which is faithful on $H_\infty\subset B$. Due to Proposition 8 it is enough to find for every $\varepsilon>0$ and every finite $K\subset H_\infty$ matrices $A$ and $B$ with properties:

1) $\|\omega(\{B^{-i}AB^i\})-1\|\geq 1$ for every nontrivial $\omega\in K$

2) $\Ad_{B^{-1}AB}A=_{O(\varepsilon)} A^2$

Put $k_0=max(K)$, $N=[\frac{1}{\varepsilon}]$. Let $j$ be natural number such that $2j+1=(2k_0+1)(2N+1)$. Consider almost representation $\varphi: H_\infty\to U_n$ from the Proposition 22. Due to condition 4) of Proposition 22 there exists unitary $u\in U_n$ such that $\Ad_u\varphi(a_{-j})=_\varepsilon\varphi(a_{j+1})$. Applying Proposition 23 to $u$ we get matrices $u_{-N},u_{-N+1},...,u_N$. Let us construct matrices $v_{-j},v_{-j+1},...,v_j$ in the following way: $v_{-j}=v_{-j+1}=...=v_{-j+2k_0}=u_{-N}$, $v_{-j+2k_0+1}=...=v_{-j+4k_0+1}=u_{-N+1}$,...,$v_{j-2k_0}=...=v_{j}=u_N$. More precisely $v_i=u_{\lfloor\frac{i+j}{2k_0+1}\rfloor-N}$.
Let us put
$$A=\Ad_{v_{-j}}\varphi(a_{-j})\oplus...\oplus\Ad_{v_j}\varphi(a_j)$$
$$B=id\otimes T^*_{2j+1}$$
where $A,B\in U(M_n\otimes M_{2j+1})=U(M_{n(2j+1)})$, $B$ is shift matrix which permutate blocks in block structure of matrix $A$, i.e. $$B^{-1}AB=\Ad_{v_{-j+1}}\varphi(a_{-j+1})\oplus...\oplus\Ad_{v_j}\varphi(a_j)\oplus \Ad_{v_{-j}}\varphi(a_{-j})$$.
 Let $P_i$ be projections corresponding to block structure of $A$ i.e. for which $P_iAP_i=0\oplus...\oplus 0\oplus \Ad_{v_i}\varphi(a_i)\oplus 0 \oplus...\oplus 0$. For notation convenience we may think that $P_iCP_i\in M_n$ for every matrix $C\in M_{n(2j+1)}$ and in these terms $P_iAP_i=\Ad_{v_i}\varphi(a_i)$.

Let us check property $\Ad_{B^{-1}AB}A=_{O(\varepsilon)}A^2$. Put $R=\Ad_{B^{-1}AB}A-A^2$. If $i\neq j$ then since $v_i=_{4\varepsilon}v_{i+1}$ and $\Ad_{\varphi(a_{i+1})}\varphi(a_i)=_\varepsilon \varphi(a_i)^2$ we have inequality $\|P_iRP_i\|=\|\Ad_{v_{i+1}^{-1}\varphi(a_{i+1})v_{i+1}}(v_i^{-1}\varphi(a_i)v_i)-v_i^{-1}\varphi(a_i)^2v_i\|\leq 17\varepsilon$. If $i=j$ then since $\Ad_{\varphi(a_{j+1})}\varphi(a_j)=_\varepsilon\varphi(a_j)^2$, $u=v_{-j}v_j^{-1}$ and $\Ad_u\varphi(a_{-j})=_\varepsilon\varphi(a_{j+1})$ we have $\|P_jRP_j\|=\|\Ad_{v_{-j}^{-1}\varphi(a_{-j})v_{-j}}(v_j^{-1}\varphi(a_j)v_j)-v_j^{-1}\varphi(a_j)^2v_j\|\leq 3\varepsilon$. Since $R$ has the same block structure as $A$ it is easy to see that $R=\oplus P_iRP_i$ and $\Ad_{B^{-1}AB}A=_{17\varepsilon} A^2$.

Let us check property $\|\omega(\{B^{-i}AB^i\})-1\|\geq 1$ for all nontrivial $\omega\in K\subset H_{k_0}\subset H_\infty$. Arbitrary $\omega$ has the following form $\omega(\{B^{-i}AB^i\})=B^{-i_1}AB^{i_1}B^{-i_2}AB^{i_2}...B^{-i_l}AB^{i_l}$ with $|i_m|\leq k_0$. But since $B^{-i}AB^i$ has the same block structure as $A$, $P_0(B^{-i}AB^i)P_0=\varphi(a_i)$ and $v_{-k_0}=....=v_k=u_0=1$ we have that
$$\|\omega(\{B^{-i}AB^i\})-1\|\geq\|P_0\omega(\{B^{-i}AB^i\})P_0-P_0\|=\|\omega(\{\varphi(a_i)\})-1\|\geq 1$$
This finishes the proof.
\end{proof}

\begin{remark}
We think that it would be interesting to examine MF-properties for groups of the form $\langle a,b|\omega(a,a^b)=1\rangle$ where group $\langle a,b|\omega(a,b)=1\rangle$ is not very difficult. If we trying to follow similarly way we see the main difficulties in construction asymptotic representation of $\langle a,b|\omega(a,b)=1\rangle$ for which spectrums of generators are almost uniform and checking MF-property for groups of the form $\langle a,b|\omega(a,b)=1\rangle \underset{\mathbb{Z}}{\star}...\underset{\mathbb{Z}}{\star}\langle a,b| \omega(a,b)=1\rangle$. In our case of Baumslag group the second problem was not considerable due to theorem of Azarov and Tieudjo and nilpotence of group $\langle a,b|b^{-1}ab=a^2\rangle$.
\end{remark}

\begin{remark}
We use in proof exotic matrix size $2^{p_n}-1$ because of simplicity of spectrum of corresponding shift matrix. But we think it is very interesting to examine uniform properties for unitary matrix $T$ for which $T^{-1}D_nT=D_n^2$. For example we think that it is important to compute $\lim_{n\to \infty}\frac{\#\{x\in \sigma(T)| x\in(a,b)\}}{n}$ for segment of circle $(a,b)\subset S^1$.
\end{remark}

{\bf Acknowledgement.} The author is grateful to V.M. Manuilov for fruitful discussions and advices.

\end{document}